\documentclass[12pt]{article}
\usepackage{amsmath,amsfonts,amssymb,amscd,verbatim,comment, amsthm}
\usepackage{fullpage}
\usepackage{color}
\usepackage{array}
\usepackage{enumerate}

\usepackage{scrextend}

\numberwithin{equation}{section}
\theoremstyle{plain}
 
\newtheorem{theorem}[equation]{Theorem}

\newtheorem{lemma}[equation]{Lemma}

\theoremstyle{definition}
\newtheorem{definition}[equation]{Definition}

\title{Multicolor Ramsey numbers for Berge cycles}
\author{Zachary DeStefano\thanks{NYU Department of Computer Science, Courant Institute, \texttt{zd2131@nyu.edu}} \and Hannah Mahon\thanks{Georgia Tech Research Institute, \texttt{hannah.mahon@gtri.gatech.edu}} \and Frank Simutis\thanks{Villanova University Department of Mathematics \& Statistics, \texttt{fsimutis@villanova.edu}.} \and Michael Tait\thanks{Villanova University Department of Mathematics \& Statistics, \texttt{michael.tait@villanova.edu}. Research is partially supported by National Science Foundation grant DMS-2011553.}}
\date{\today}

\begin{document}

\maketitle

\begin{abstract}
    In this paper, for small uniformities, we determine the order of magnitude of the multicolor Ramsey numbers for Berge cycles of length $4$, $5$, $6$, $7$, $10$, or $11$. Our result follows from a more general setup which can be applied to other hypergraph Ramsey problems. Using this, we additionally determine the order of magnitude of the multicolor Ramsey number for Berge-$K_{a,b}$ for certain $a$, $b$, and uniformities.
\end{abstract}

\section{Introduction}

Given a family of hypergraphs $\mathcal{F}$, the $k$-color Ramsey number for $\mathcal{F}$ is the minimum $n$ such that for any edge coloring of the complete $r$-uniform hypergraph on $n$ vertices with $k$ colors, we have that there exists a monochromatic subgraph $F$ for some $F\in \mathcal{F}$. We will denote this quantity by $R_r(\mathcal{F};k)$. The study of graph and hypergraph Ramsey numbers represents a huge body of research, and we refer the reader to the surveys \cite{CFS} and \cite{MS}. 

In this paper we will be interested in hypergraph Ramsey numbers where the number of colors goes to infinity. We will focus on families of hypergraphs which are Berge-$G$ for some graph $G$, defined as follows. Given a ($2$-uniform) graph $G$, we say that a hypergraph $H$ is a {\em Berge-$G$} if $V(G) \subset V(H)$ and there is a bijection $\phi:E(G) \to E(H)$ such that $e\subset \phi(e)$ for all $e \in E(G)$. In other words, $H$ is a Berge-$G$ if we can embed a single edge into each hyperedge of $H$ and create a copy of $G$. When $G$ is a path or cycle, this definition agrees with the definition of a Berge path or Berge cycle. Note that many nonisomorphic hypergraphs may be a Berge-$G$, and we denote the family of all such hypergraphs by $\mathcal{B}(G)$. The notion of the family of Berge-$G$ for general graphs $G$ was initiated in \cite{GP} and since then extensive research has been done on extremal problems related to $\mathcal{B}(G)$ for various graphs $G$.

The Tur\'an number for a family $\mathcal{F}$ is denoted by $\mathrm{ex}_r(n, \mathcal{F})$ and is the maximum number of edges in an $n$-vertex $r$-uniform hypergraph that does not contain any $F\in \mathcal{F}$ as a subgraph. Early work on extremal problems for Berge hypergraphs focused on Tur\'an numbers of $\mathcal{B}(G)$. Since the introduction of the Berge-Tur\'an problem, a long list of papers have been written about it, far too many to cite here, and we recommend \cite{GMP} for a partial history. More recently Ramsey problems have also been considered, see for example \cite{AG, BGKW, BZ, G, GMOV, GLSS, GS, LW, NV, P, STWZ}. The two problems are related as any coloring avoiding monochromatic $\mathcal{B}(G)$ must have that every color class contains at most $\mathrm{ex}_r(n, \mathcal{B}(G))$ edges.  It is therefore not surprising that the order of magnitude for $R_2(C_{2m};k)$, the multicolor Ramsey number of an even cycle, is known only when $k\in \{2,3,5\}$. In these cases, Li and Lih \cite{LL} showed that $R_2(C_{2m};k) = \Theta\left( k^{\frac{m}{m-1}}\right)$. Our main result is a generalization of this to hypergraphs. We prove our main result as a corollary of some more general theorems which may be useful for future hypergraph Ramsey problems.

Lower bounds for $R_r(\mathcal{F}; k)$ may be proved by considering the dual problem of minimizing the number of colors necessary to partition the edge set of $K_n^{(r)}$ such that each color class is $\mathcal{F}$-free. Our first theorem reduces this dual problem to covering the edges of complete $r$-partite $r$-uniform hypergraphs. We use $K_n^{(r)}$ to denote the complete $r$-uniform hypergraph on $n$ vertices and $K_{n,\cdots, n}^{(r)}$ to denote the complete $r$-partite $r$-uniform hypergraph with $n$ vertices in each part.

\begin{theorem} \label{thm: uniform to complete}
Let $r$ be fixed and $\beta>r-2$ and let $\mathcal{F}$ be a family of connected hypergraphs. If there exists an edge coloring of $K_{n,\cdots, n}^{(r)}$ with $O(n^\beta)$ colors with no monochromatic $F \in \mathcal{F}$, then there exists a coloring of the edges of $K_n^{(r)}$ with $O(n^\beta)$ colors with no monochromatic $F \in \mathcal{F}$.
\end{theorem}

We use this theorem to prove our main result, which determines the order of magnitude of the multicolor Ramsey number for Berge cycles of certain lengths and certain uniformities.

\begin{theorem}\label{thm: cycles ramsey number}
For $m\in \{2,3,5\}$, if $r<4m-1$, then $R_r(\mathcal{B}(C_{2m});k)$ and $R_r(\mathcal{B}(C_{2m+1});k)$ are each $\Theta \left(k^{\frac{m}{rm-m-1}}\right)$
\end{theorem}

We note that our proof yields that if one could determine that the order of magnitude of the graph multicolor Ramsey number for $C_{2m}$ is $\Theta(k^{\frac{m}{m-1}})$ for some $m\not\in \{2,3,5\}$ then this would also determine that for all $r < 4m-1$ the order of magnitude of the $r$-uniform multicolor Ramsey number for $\mathcal{B}(C_{2m})$ and for $\mathcal{B}(C_{2m+1})$ is $\Theta \left(k^{\frac{m}{rm-m-1}}\right)$. Using similar techniques, we are also able to give lower bounds on $R_r(\mathcal{B}(K_{a,b});k)$ for some choices of $r, a, b$. 

\begin{theorem}\label{thm: biclique ramsey number}
Let $b\geq 2$ and $a > (b-1)!$. Then for all $r < 2(a+b)-1$ we have 
\[
R_r(\mathcal{B}(K_{a,b});k) = \Omega\left(k^{\frac{b}{(r-2)b+1}} \right).
\]
Furthermore, when $b=2$ or $a+b \leq r < 2(a+b)-1$, for $a> (b-1)!$ we have 
\[
R_r(\mathcal{B}(K_{a,b});k) = \Theta\left(k^{\frac{b}{(r-2)b+1}} \right).
\]
\end{theorem}


\section{Preliminaries}

\begin{definition}
For a given $r$, there are a finite number of possible vectors $(\rho_1,...,\rho_r)$ such that $\rho_i \in \mathbb{N}\cup \{0\}$ and $\sum \rho_i = r$. We will call the set of these vectors $P_r$. Given a particular vector $\pmb{\rho} \in P_r$, we have the following shorthand for describing specific features of this vector.
The maximal element is $\pmb{\rho}_{max}$.
The number of non-zero $\rho_i$ is $|\text{supp}(\pmb{\rho})|$ which we will notate as $\pmb{\hat{\rho}}$ for brevity.
\end{definition}

\begin{definition}
We define $(P_r,\prec)$ to be the partial ordering of these weak compositions of $r$ as the following. $\forall \pmb{\rho},\pmb{\tau} \in P_r$, $\pmb{\rho} \prec \pmb{\tau}$ if $\pmb{\hat{\rho}}> \pmb{\hat{\tau}}$ and there exists an ordering of a partition of $\pmb{\rho}$ into $\pmb{\hat{\tau}}$ subsets such that the sum of the elements in the $i$'th ordered subset of $\pmb{\rho}$ are equal to the $i$'th nonzero entry of $\pmb{\tau}$ for all $i$. \end{definition}

By considering any linear extension of the poset $(P_r,\prec)$, we arrive at a total ordering of $P_r$ with smallest element $(1,1,...,1)$ which we can then induct over.

\begin{definition}
We define $H^{(r)}_{(\rho_1,...,\rho_r)}(n)$ to be the hypergraph with vertex set $V = V_1 \cup V_2 \cup ... \cup V_r$ where $|V_i|=n$ and edge set $\{e : |e \cap V_i|=\rho_i \}$.
\end{definition}

We will need the following procedure which takes a graph and transforms it to a hypergraph of higher uniformity.

\begin{definition}[Enlarging] 
Let $G$ be a bipartite graph with partite sets $A$ and $B$ and let $a,b\in \mathbb{N}$. Define an $(a+b)$-uniform hypergraph $H$ as follows. For each $v\in A$ let $v_1,\cdots, v_a$ be $a$ disjoint vertices and for each $u\in B$ let $u_1,\cdots, u_b$ be $b$ disjoint vertices. Then 
\[
V(H) = \left(\bigcup_{v\in A} \{v_1,\cdots , v_a\}\right) \cup \left( \bigcup_{u \in B} \{u_1, \cdots, u_b\}\right),
\]
\[
E(H) = \left \{ \{v_1,\cdots, v_a, u_1,\cdots, u_b\}: uv\in E(G) \right\}.
\]

We say that $H$ {\em is the hypergraph obtained by enlarging each vertex in $A$ to $a$ vertices and each vertex in $B$ to $b$ vertices}.
\end{definition}

As stated in the introduction, determining the minimum $n$ such that any coloring of $K_n^{(r)}$ has a monochromatic $F$ is equivalent to the dual problem of minimizing the number of colors necessary to color $K_n^{(r)}$ such that no color class contains an $F$. We formalize this with the following function. 
\begin{definition}
Let $H$ be a hypergraph and $\mathcal{F}$ be a family of hypergraphs. Define the function $C(H, \mathcal{F})$ to be the minimum number of colors necessary to color the edge set of $H$ such that no color class contains any $F\in \mathcal{F}$.
\end{definition}
\section{Proof of Theorem \ref{thm: uniform to complete}}
Let $\beta>r-2$ and let $\mathcal{F}$ be a fixed family of connected hypergraphs, and assume that we can color the edges of complete $r$-uniform $r$-partite hypergraph with $O(n^{\beta})$ colors so that there is no monochromatic copy of a hypergraph in $\mathcal{F}$. That is, there exists a constant $c_{1,\cdots, 1}$ such that $C(H_{1,\cdots, 1}^{(r)}(n), \mathcal{F}) \leq c_{1,\cdots, 1} n^{\beta}$ for all $n$. We aim to show that $C(K_n^{(r)}, \mathcal{F}) = O(n^{\beta})$. To do this, we will split the edge set of $K_n^{(r)}$ into a bounded number of parts each associated to an element of the poset $P_r$ and show that that each of these sets can be colored with $O(n^{\beta})$ colors. 

Since $C(K_n^{(r)}, \mathcal{F})$ is monotone in $n$, we assume without loss of generality that $n$ is divisible by $r$. Divide the vertex set into $V_1,..., V_r$ each of size $\frac{n}{r}$. For each edge $e$ there is a vector $(e_1,...,e_r) \in P_r$ where $e_i = |e\cap V_i|$, and we may partition the edge set of $K^{(r)}_n$ into sets depending on which vector in $P_r$ it is associated with. For a given vector $\pmb{\rho}\in P_r$ the set of edges with vector $\pmb{\rho}$ forms a subhypergraph isomorphic to $H_{\pmb{\rho}}^{(r)}\left(\frac{n}{r}\right)$, and hence

\[K_n^{(r)} = \bigcup\limits_{\pmb{\rho} \in P_r} H_{\pmb{\rho}}^{(r)}\left(\frac{n}{r}\right).\]

Since the number of vectors in $P_r$ is a constant that depends only on $r$, it suffices to show that for each $\pmb{\rho}\in P_r$ we have that $C(H_{\pmb{\rho}}^{(r)}, \mathcal{F}) = O(n^{\beta})$. We will proceed by induction on (any linear extension of) $P_r$. Since by the assumption we have that $C(H_{1,\cdots, 1}^{(r)}(n), \mathcal{F}) \leq c_{1,\cdots, 1} n^{\beta}$, the base case is satisfied. Now fix $\pmb{\rho} = (\rho_1,\cdots, \rho_r)\in P_r$ and assume that for all $\pmb{\tau} \prec \pmb{\rho}$ there is a constant $c_{\pmb{\tau}}$ such that $C(H_{\pmb{\tau}}^{(r)}(n), \mathcal{F}) \leq c_{\pmb{\tau}} n^{\beta}$ for all $n$. 

Note that if $\rho_i = 0$, then $V_i$ is not incident with any hyperedges of $H_{\pmb{\rho}}^{(r)}(n)$. Without loss of generality we can assume that $\rho_1$ through $\rho_{\pmb{\hat{\rho}}}$ are non-zero. Split each $V_i$ where $\rho_i > 0$ into $\pmb{\rho}_{max}$ parts $V_{i, 1},\cdots, V_{i,\pmb{\rho}_{max}}$ (again without loss of generality assume that $n$ is divisible by $\pmb{\rho}_{max}$). Divide the edges of $H_{\pmb{\rho}}^{(r)}\left(n\right)$ as follows.  Call an edge $e$ {\em Type I} if for all $i$ there exists a $j$ such that $e\cap V_{i,j} = e\cap V_i$. Call the other edges {\em Type II}. We will show that we may cover the Type I and Type II edges with $O(n^{\beta})$ $\mathcal{F}$-free hypergraphs by induction on $n$ and by the induction hypothesis on $P_r$ respectively. 

First we take care of the Type II edges. For any choice $U_1,\cdots, U_r$ of distinct sets from $\{V_{i,j}\}_{i,j}$ we may consider the subhypergraph of Type II edges which are induced by $U_1,\cdots, U_r$. If this subhypergraph contains edges, then for each edge $e$ one may consider the vector $(e'_1,\cdots, e'_r)$ where $e'_i = |U_i \cap e|$. By definition of Type II, the vector $(e'_1,\cdots, e'_r)$ is strictly less than $\pmb{\rho}$ in $P_r$. Therefore, by the induction hypothesis (on $P_r$), this subhypergraph of edges may be covered by $O(n^\beta)$ hypergraphs each of which is $\mathcal{F}$-free. Since the number of choices for $U_1,\cdots, U_r$ is a constant that depends only on $r$ and $\pmb{\rho}_{max}$, we have that there is an absolute constant $C:=C_{r, \pmb{\rho}}$ so that the Type II edges may be covered with at most $C n^{\beta}$ $\mathcal{F}$-free subhypergraphs.

Next we take care of the Type I edges by induction on $n$. Define $C_1$ to be a constant that satisfies $C + C_1 \pmb{\rho}_{max}^{\pmb{\hat{\rho}} - 1 - \beta} < C_1$. This is possible since $\beta > r-2$ and $\pmb{\hat{\rho}} \leq r-1$ for any $\pmb{\rho} \not= (1,\cdots, 1)$. For the induction hypothesis, assume that for any $k < n$ we have that $C(H_{\pmb{\rho}}^{(r)}(k), \mathcal{F}) \leq C_1 k^{\beta}$. For any $\mathbf{j} = (j_1,\cdots, j_{\hat{\pmb{\rho}}}) \in \{1,\cdots, {\pmb{\rho}_{max}}\}^{\hat{\pmb{\rho}}}$ the graph of Type I edges induced by $V_{1, j_1},\cdots, V_{{\hat{\pmb{\rho}}}, j_{\hat{\pmb{\rho}}}}$ is isomorphic to $H_{\pmb{\rho}}^{(r)}\left(\frac{n}{\pmb{\rho}_{max}}\right)$. By the induction hypothesis (on $n$) there are $\mathcal{F}$-free hypergraphs $G_1(\mathbf{j}),\cdots, G_{T}(\mathbf{j})$ which cover the Type I edges induced by $V_{1, j_1},\cdots, V_{{\hat{\pmb{\rho}}}, j_{\hat{\pmb{\rho}}}}$ where $T = C_1\left( \frac{n}{ \pmb{\rho}_{max}}\right)^\beta$. Naively, we could use such a set of hypergraphs for each $\mathbf{j}$, but unfortunately this is not a small enough number in total. In order to reduce the total number of hypergraphs used, we will combine those which are edge-disjoint. Note that because $\mathcal{F}$ contains only connected hypergraphs, the disjoint union of $\mathcal{F}$-free graphs is still $\mathcal{F}$-free.

For each $\mathbf{j}$ assume that we have $\mathcal{F}$-free hypergraphs $G_1(\mathbf{j}),\cdots, G_T(\mathbf{j})$ which partition the Type I edges induced by $V_{1, j_1},\cdots, V_{{\hat{\pmb{\rho}}}, j_{\hat{\pmb{\rho}}}}$ and $T = C_1\left( \frac{n}{ \pmb{\rho}_{max}}\right)^\beta$. We combine disjoint copies of these as follows. For $k_2,\cdots, k_{\pmb{\hat{\rho}}}$ any ${\pmb{\hat{\rho}}}-1$ (not necessarily distinct) integers in $\{0,\cdots, \pmb{\rho}_{max}-1\}$, consider the vectors $\mathbf{j}_1 = (1, 1+k_2,\cdots, 1+k_{\pmb{\hat{\rho}}})$, $\mathbf{j}_2 = (2, 2+k_2, \cdots, 2+k_{\pmb{\hat{\rho}}})$, .... $\mathbf{j}_{\pmb{\hat{\rho}}} = ({\pmb{\hat{\rho}}}, {\pmb{\hat{\rho}}}+k_2, \cdots, {\pmb{\hat{\rho}}}+k_{\pmb{\hat{\rho}}})$ where addition is done on $\{1,\cdots, \pmb{\rho}_{max}\}$ mod $\pmb{\rho}_{max}$. Then for any $t$ the graphs $G_t(\mathbf{j}_1), \cdots ,G_t(\mathbf{j}_{\pmb{\hat{\rho}}})$ are disjoint. Let their union be called $G_t(k_2,\cdots, k_{\pmb{\hat{\rho}}})$. Then as $k_2,\cdots, k_{\pmb{\hat{\rho}}}$ vary we have 
\[
\bigcup_{t=1}^T \bigcup_{k_2,\cdots, k_{\pmb{\hat{\rho}}}} G_t(k_2,\cdots, k_{\pmb{\hat{\rho}}}) =  \bigcup_{t=1}^T \bigcup_{\mathbf{j}} G_t(\mathbf{j}), 
\]
and this union covers all of the Type I edges. The total number of graphs $G_t(k_2,\cdots, k_{\pmb{\hat{\rho}}})$ is $\pmb{\rho}_{max}^{\pmb{\hat{\rho}}-1} \cdot T = C_1 \left(\pmb{\rho}_{max}^{\pmb{\hat{\rho}} -1 - \beta }\right)n^\beta$. Combining the graphs used to cover the Type I edges with the graphs used to cover the Type II edges we have that 
\[
C(H_{\pmb{\rho}}^{(r)}(n), \mathcal{F}) \leq Cn^{\beta} + C_1 \left(\pmb{\rho}_{max}^{{\pmb{\hat{\rho}}}-1 - \beta }\right)n^\beta < C_1 n^{\beta},
\]
where the last inequality follows by the choice of $C_1$.

\section{Proof of Theorems \ref{thm: cycles ramsey number} and \ref{thm: biclique ramsey number}}
We need the following lemmas which take a graph and transform it to a hypergraph which forbids something. Lemma \ref{lem: cycle blowup} has been noted before, see Construction 1.9 in \cite{GL2} for example, but we include a proof for completeness.

\begin{lemma}\label{lem: cycle blowup}
Let $G$ be a bipartite graph with no $C_3, C_4,...,C_{2m},C_{2m+1}$. Let $H$ be the $(s+t)$-uniform hypergraph obtained by enlarging each vertex in one part of $G$ to $s$ vertices and each vertex in the other part of $G$ to $t$ vertices. Then if $s < 2m$ and $t < 2m$, $H$ is $\mathcal{B}(C_{2m})$ and $\mathcal{B}(C_{2m+1})$ free.
\end{lemma}

\begin{proof}
By contrapositive, let $g\in \{2m, 2m+1\}$ and assume that $H$ contains a Berge-$C_g$ with vertex set $v_1,...,v_{g}$ and edge set $e_1,...,e_{g}$ such that $v_i,v_{i+1} \in e_i$ (subscripts considered modulo $g$). Let $A$ and $B$ be the partite sets of graph $G$. For each $v_j$ let $w_j$ be the vertex in $G$ which was enlarged to create $v_j$. Note that the $w_j$ may not be distinct, but in the sequence $(w_1,\cdots, w_t)$ a vertex may appear at most $s$ times if it is in $A$ and at most $t$ times if it is in $B$. For each $j$, if $w_j$ and $w_{j+1}$ are distinct, then $w_j \sim w_{j+1}$ in $G$. Then, ignoring repeated vertices, the sequence $w_1, w_2, \dots, w_g, w_1$ corresponds to a closed walk in $G$ of length $\ell \leq g$. Furthermore, since $e_1,\cdots, e_g$ are distinct hyperedges, the edges in this closed walk must be distinct. Since $g \geq 2m$ and $s < 2m$ and $t<2m$, we have that $\ell \geq 3$ Therefore, there is a cycle in $G$ of length at least $3$ and at most $g$.

\end{proof}

We prove a similar lemma regarding enlarging graphs that are $K_{a,b}$-free.

\begin{lemma}\label{lem: biclique blowup}
Let $a,b\geq 2$ and $G$ be a bipartite graph with no $K_{a,b}$. Let $H$ be the $(s + t)$-uniform hypergraph obtained by enlarging each vertex in one part of $G$ to $s$ vertices and each vertex in the other part of $G$ to $t$ vertices. Then if $s<a+b$ and $t<a+b$, $H$ does not contain a Berge-$K_{a,b}$.
\end{lemma}

\begin{proof}
By contrapositive, assume $H$ contains a Berge-$K_{a,b}$ with vertex set $v_1,\cdots, v_a, u_1,\cdots, u_b$ and edge set $\{e_{i,j}\}$ where $\{v_i, u_j\} \subset e_{i,j}$. Let the partite sets of $G$ be $A'$ and $B'$ and let $A$ be the set of vertices that came from enlarging vertices in $A'$ and $B$ be the set of vertices that came from enlarging vertices in $B'$.  For each $u_i$ and $v_j$, let $u'_i$ and $v'_j$ be the vertex in $G$ that was enlarged to create $u_i$ or $v_j$ respectively. First the set $\{v'_1,\cdots, v'_a, u'_1,\cdots, u'_b\}$ contains more than one vertex because each vertex of $G$ was enlarged to either $s$ or $t$ vertices in $H$ and both $s$ and $t$ are at most $a+b-1$ by assumption. Therefore, there exist $u'_i$ and $v'_j$ that are adjacent in $G$, and we may assume without loss of generality that $v'_j \in A'$ and $u'_i \in B'$ (and therefore $v_j\in A$ and $u_j \in B$). 

Next we will show that $v_k$ and $u_k$ are in $A$ and $B$ respectively for all $k$. The only vertices in $A$ that $v_i$ shares edges with are those that came from enlarging $v'_i$. Therefore, if $u_k \in A$ for some $k$ we must have that $u'_k = v'_i$. But then this forces all vertices all vertices in $A$ to come from enlarging either $v'_i$ or $u'_j$. For $a>1$ this is a contradiction for then the map from the edges of the Berge-$K_{a,b}$ to the edges of $K_{a,b}$ will not be a bijection. A similar contradiction occurs if $v_k\in B$ for some $k$.

Since all $v_i$ are in $A$ and $u_j$ are in $B$, we must also have that all $v'_i$ and all $v'_j$ are distinct, otherwise again, since $a,b\geq 2$, the map  from the edges of the Berge-$K_{a,b}$ to the edges of $K_{a,b}$ will not be a bijection. But now if all $v'_i$ and $u'_j$ are distinct, we have that $v'_i$ and $u'_j$ are adjacent in $G$ for all $i$ and $j$, ie there is a $K_{a,b}$ in $G$.

\end{proof}

We will also use the following general theorem which allows one to obtain a coloring of $K_{n,\cdots, n}^{(r)}$ given a coloring of $K_{n,n}$.

\begin{theorem}\label{thm: hypergraph partitioning}
Let $G_1,\cdots, G_T$ be bipartite graphs on partite sets $A$ and $B$ whose union is $K_{n,n}$. For each $j$ let $H_j$ be the hypergraph obtained by enlarging each vertex in $A$ to $s$ vertices and each vertex in $B$ to $t$ vertices. Assume that $\mathcal{F}$ is a family of hypergraphs such that $H_i$ is $\mathcal{F}$-free for all $i$. Then there is a partition of the edge set of the complete $(s+t)$-partite $(s+t)$-uniform hypergraph with $n$ vertices in each part into $T \cdot n^{s + t - 2}$ subgraphs each of which is $\mathcal{F}$-free.
\end{theorem}

\begin{proof}
Let $A$ and $B$ be identified with $\mathbb{Z}/ n\mathbb{Z}$, and let $A_1,\cdots, A_s$ and $B_1,\cdots, B_t$ be disjoint sets of vertices also each identified with $\mathbb{Z} /n\mathbb{Z}$. For $a_2, \cdots, a_s$ and $b_2, \cdots, b_t$ arbitrary elements of $\mathbb{Z} / n\mathbb{Z}$ and $1\leq j\leq T$, define the $(s+t)$-partite $(s+t)$-uniform hypergraph $H_i (a_2, \cdots, a_s, b_2, \cdots, b_t)$ to be the $(s+t)$-partite $(s+t)$-uniform hypergraph on partite sets $A_1,\cdots, A_s, B_1,\cdots, B_t$ with edge set 
\[
\{(u, u+a_2, \cdots, u+a_s, v, v+b_2, \cdots v+b_t): uv\in E(G_i) \},
\]
where vertices in coordinates $1$ through $s$ are in parts $A_1,\cdots, A_s$ respectively and vertices in coordinates $s+1,\cdots s+t$ are in parts $B_1,\cdots, B_t$ respectively.

Note that $H_i(0, \cdots, 0, 0,\cdots, 0)$ is isomorphic to the hypergraph obtained by enlarging each vertex in $G_i$ to $s$ vertices if it is in $A$ and to $t$ vertices if it is in $B$, and hence it is $\mathcal{F}$-free. Furthermore for any choice $a_2,\cdots, a_s, b_2, \cdots b_t$, the hypergraph $H_i(a_2,\cdots, a_s, b_2, \cdots b_t)$ is isomorphic to $H_i(0,\cdots, 0, 0, \cdots, 0)$ via the explicit automorphism
\[ u\mapsto
\begin{cases}
u & u\in A_1 \cup B_1\\
u+a_i & u\in A_i, i\geq 2\\
u+b_i & u\in B_i, i\geq 2
\end{cases}
\]
Note that as $i$ ranges from $1$ to $T$ and $a_2,\cdots, a_s, b_2, \cdots b_t$ vary over all choices in $\mathbb{Z} / n\mathbb{Z}$, we have $T\cdot n^{s+2-2}$ hypergraphs $H_i(a_2,\cdots, a_s, b_2, \cdots b_t)$. It only remains to show that 
\[
\bigcup_{j=1}^T \bigcup_{a_2,\cdots, a_s, b_2, \cdots b_t} H_i(a_2,\cdots, a_s, b_2, \cdots b_t)
\]
covers all of the hyperedges of the complete $r$-partite $r$-uniform hypergraph on partite sets $A_1, \cdots, A_s, B_1,\cdots, B_t$. To do this, consider an arbitrary hyperedge $(v_1,\cdots, v_{s+t})$. Let $i$ be the index such that $v_1 v_{s+1} \in E(G_i)$ (this is well-defined since the union of $G_1,\cdots, G_T$ is $K_{n,n}$). Then by the definitions we have that $(v_1,\cdots, v_{s+t})$ is an edge of the hypergraph
\[
H_i(v_2-v_1, v_3-v_1, \cdots, v_s - v_1, v_{s+2} - v_{s+1}, \cdots, v_{s+t} - v_{s+1}).
\]
\end{proof}

\begin{proof}[Proof of Theorem \ref{thm: cycles ramsey number}]
It is known \cite{GL} that the Tur\'an numbers for Berge cycles satisfy
\begin{align*}
\mathrm{ex}_r(n, \mathcal{B}(C_{2m})) &= O\left( n^{1+\frac{1}{m}}\right)\\
\mathrm{ex}_r(n, \mathcal{B}(C_{2m+1})) &= O\left( n^{1+\frac{1}{m}}\right)
\end{align*}

Applying this result and the pigeonhole principle yields the upper bound. For the lower bound, showing that $R_r(\mathcal{B}(C_{2m}); k)$ and $R_r(\mathcal{B}(C_{2m+1}); k)$ are $\Omega\left(k^{\frac{m}{rm-m-1}}\right)$ is equivalent to showing that $K_n^{(r)}$ can be partitioned into $O\left(n^{r - 1 - \frac{1}{m}} \right)$ subgraphs each of which are $\mathcal{B}(C_{2m})$ and $\mathcal{B}(C_{2m+1})$ free respectively. Let $s$ and $t$ be defined so that $s+t = r$ and $s,t < 2m$.

It is known that for $m\in \{2,3,5\}$, $K_{n,n}$ can be partitioned into $O\left(n^{1-\frac{1}{m}}\right)$ subgraphs each of which has girth at least $2m+2$ (see Lemma 5 of \cite{LL} for the case $m=2$ and Proposition 3.1 of \cite{T} for the cases when $m=3$ and $m=5$). Therefore, for $T = O\left(n^{1 - \frac{1}{m}}\right)$ assume that $G_1, \cdots, G_T$ are graphs each of which have girth at least $2m+2$ and whose union is $K_{n,n}$. By Lemma \ref{lem: cycle blowup}, for each $G_i$ the hypergraph obtained by enlarging each vertex in one partite set to $s$ vertices and each vertex in the other partite set to $t$ vertices is both $\mathcal{B}(C_{2m})$-free and $\mathcal{B}(C_{2m+1})$-free. Then, by applying Theorem \ref{thm: hypergraph partitioning}, we have a set of $O\left(n^{r-1-\frac{1}{m}}\right)$ subgraphs which are $\mathcal{B}(C_{2m})$ and $\mathcal{B}(C_{2m+1})$-free the union of which cover the edges of $K_{n,\cdots, n}^{(r)}$. Applying Theorem \ref{thm: uniform to complete}, we may partition the edge set of $K_n^{(r)}$ into $O\left(n^{r-1-\frac{1}{m}}\right)$ subgraphs each of which are $\mathcal{B}(C_{2m})$ and $\mathcal{B}(C_{2m+1})$-free. This completes the proof.
\end{proof}

\begin{proof}[Proof of Theorem \ref{thm: biclique ramsey number}]
The lower bound is similar to the proof of the lower bound in Theorem \ref{thm: cycles ramsey number}. We leave the details to the reader and only note that one uses Lemma \ref{lem: biclique blowup} and the result from \cite{ARS} that for $a>(b-1)!$, the edge set of $K_n$ may be partitioned into $\Theta(n^{1/b})$ subgraphs each of which is $K_{a,b}$-free.

For the upper bound, when $b=2$ we use the result from \cite{GMP} that 
\[
\mathrm{ex}_r(n, \mathcal{B}(K_{2,t})) = O\left(n^{3/2} \right),
\]
for all $r$ and $t$ and the result from \cite{GP} that 
\[
\mathrm{ex}_r(n, \mathcal{B}(K_{a,b})) = O(n^{2-1/s})
\]
whenever $r\geq a+b$. The bound then follows from the pigeonhole principle.

\end{proof}

\section{Conclusion}
In this paper we determined the order of magnitude for the multicolor Ramsey numbers of Berge cycles of length $4$, $5$, $6$, $7$, $10$, or $11$, as long as the uniformity is small enough. Extending our theorem to other cycle lengths or uniformities is out of reach at the current time, for in these cases we do not even know the order of magnitude of the Tur\'an number $\mathrm{ex}_r(n, C_\ell)$. Our main result follows from a more general set up that allows one to go from a construction in the graph case to a construction in the hypergraph case. Because of this we were also able to give the order of magnitude for $R_r(\mathcal{B}(K_{a,b}))$ for some choices of $r,a,b$. The lower bound in Theorem \ref{thm: biclique ramsey number} is not tight in general. It is known (see \cite{GMP}) that 
\[
\mathrm{ex}(n, K_r, F) \leq \mathrm{ex}_r(n, \mathcal{B}(F)) \leq \mathrm{ex}(n, K_r, F) + \mathrm{ex}(n, F),
\]
where $\mathrm{ex}(n, K_r, F)$ denotes the maximum number of copies of $K_r$ in an $n$-vertex $F$-free graph. Combining this with results from \cite{AS} gives that for $a>(b-1)!$ and $3\leq r\leq \tfrac{a}{2} + 1$,
\[
\mathrm{ex}_r(n, \mathcal{B}(K_{a,b})) = \Theta \left( n^{r - \binom{r}{2}/a}\right)
\]

The upper bound that one gets from the pigeonhole principle for such $r,a,b$ does not match our lower bound in Theorem \ref{thm: biclique ramsey number}. Perhaps one could leverage the projective norm graphs to improve on our result in these cases. When $\tfrac{a}{2} + 1 < r < a+b$ the order of magnitude for $\mathrm{ex}_r(n, \mathcal{B}(K_{a,b}))$ is not known and this would have to be determined before answering the Ramsey question. It would be interesting to determine the order of magnitude for the multicolor Ramsey number of $\mathcal{B}(G)$ for other graphs $G$.

Throughout this paper we did not try to optimize our multiplicative constants because doing so would not have given us an asymptotic formula in any of the cases. We note that in all of the constructions as they are written, there are pairs of color classes that correspond to edge disjoint hypergraphs, and these could be combined to reduce the total number of colors used. It is not clear what the best way to do this systematically is, but for example, we can obtain a lower bound for $R_3(n, \mathcal{B}(C_4))$ of \[\left(\frac{(3\sqrt{2}-4)(3\sqrt{3}-1)}{2} -o(1)\right)^{2/3}k^{2/3} \approx 0.63756 k^{2/3}.\] Furthermore, in some cases it is possible to extend Theorem \ref{thm: uniform to complete} to some $\beta \leq r-2$. Determining an asymptotic formula for any of the Ramsey numbers studied in this paper would be very interesting but would require new ideas, as even asymptotics for the corresponding Tur\'an numbers are not known (cf \cite{EGM, EGMST, GL} and Section 5 of \cite{GMP}). In the specific case of $3$-uniform graphs of girth $5$, it is known \cite{LV} that 
\[
\mathrm{ex}_3(n, \{\mathcal{B}(C_2), \mathcal{B}(C_3), \mathcal{B}(C_4)\}) \sim \frac{1}{6} n^{3/2}.
\]
One construction showing the lower bound is to take the vertex set to be the $1$-dimensional subspaces of $\mathbb{F}_q^3$ where $3$ subspaces form an edge if and only if they are an orthogonal basis. It would be interesting to try to use automorphisms of this graph to show (if it is true) that \[R_3(\{\mathcal{B}(C_2), \mathcal{B}(C_3), \mathcal{B}(C_4)\}; k) \sim k^{2/3}.\]

\bibliographystyle{plain}
\bibliography{bib.bib}

\end{document}